\documentclass{amsart}
\pdfoutput=1

\usepackage{microtype}
\usepackage{amssymb,latexsym}
\usepackage[mathscr]{eucal}
\usepackage{mathtools}
\usepackage{amsthm}
\usepackage{enumitem}
\usepackage{hyperref}
\usepackage{mleftright}

\DeclareMathOperator{\scal}{Sc}
\DeclareMathOperator{\two}{II}
\DeclareMathOperator{\tr}{tr}

\theoremstyle{plain}
\newtheorem{theorem}{Theorem}[section]
\newtheorem{proposition}[theorem]{Proposition}

\newtheorem{lemma}[theorem]{Lemma}

\theoremstyle{remark}

\theoremstyle{definition}

\newtheorem{question}[theorem]{Question}

\begin{document}
	
\title[Normal curvature of tori]{Normal curvature of immersed tori \\of dimension at most $18$}

\author{Matteo Raffaelli}
\address{School of Mathematics, Georgia Institute of Technology, Atlanta, Georgia 30332}
\email{raffaelli@math.gatech.edu}
\date{September 15, 2026}

\begin{abstract} 
For $n\leq 18$, we prove that any smooth immersion of the $n$-torus into the closed unit ball in $\mathbb R^q$ has a point at which the spherical average of $\lvert \two(v,v)\rvert^2$ is at least $3n/(n+2)$. This answers a question of Petrunin in these dimensions. The proof combines the scalar curvature obstruction for the torus with a conformal Laplacian argument, reducing the problem to a one-dimensional differential inequality. We also show that this reduction cannot yield the result in dimensions $n\geq19$.
\end{abstract}
\maketitle

\section{Introduction}

A smooth torus $T^n$ immersed in the closed unit ball $B^q \subset \mathbb R^q$ cannot be too flat: by a theorem of Petrunin~\cite{petrunin2024}, there exists a point $p\in T^n$ and a unit-speed geodesic of $T^n$ whose ambient curvature is at least $\sqrt{3n/(n+2)}$ at $p$. The bound is sharp, with equality holding along every unit-speed geodesic of Gromov’s flat tori~\cite[3.B]{gromov2022}, \cite[2.A]{gromov2023}.

In fact, the proof establishes more. Let $\kappa(p)$ be the average of $\lvert \two(v,v)\rvert^2$ over unit vectors in $T_p T^n$, where $\two$ denotes the second fundamental form. Petrunin showed that if $n\leq 4$, or if $\lvert \two(v,v)\rvert \leq2$ for all unit vectors in the tangent bundle $TT^n$, then $\max\kappa \geq 3n/(n+2)$. Informally, there is a point at which the immersion is sufficiently bent not merely in one direction, but in an averaged sense over all directions. It is natural to ask whether the same conclusion holds without either assumption.

\begin{question}[{\cite[Q.~1.3]{petrunin2024}}]\label{q:petrunin}
Is it true that for every smooth immersion $T^n\looparrowright B^q$, the inequality $\kappa(p) \geq 3n/(n+2)$ holds at some point $p\in T^n$?
\end{question}

Here we answer this question affirmatively in all dimensions $n\leq 18$.

\begin{theorem}\label{thm:main}
Let $T^n\looparrowright B^q$ be a smooth immersion. If $1\leq n\leq 18$, then there exists a point $p\in T^n$ such that $\kappa(p) \geq 3n/(n+2)$.
\end{theorem}

The proof, given in section~\ref{sec:proof}, builds on Petrunin's conformal Laplacian argument, which yields the result for $n\leq 4$. The main new idea is to allow the radial test function $u$ to vary rather than fixing it in advance. Using the spherical averaging identity, the Gauss equation, and the confinement of the immersion to the unit ball, we reduce the condition $Lu>0$, where $L$ denotes the conformal Laplacian, to a one-dimensional problem: finding a function that remains sufficiently small while growing sufficiently rapidly.

The dimension bound $n\leq 18$ is not an artifact of the particular ansatz used in the proof. In section~\ref{sec:18}, we show that for $n\geq 19$ the resulting one-dimensional inequalities cannot be satisfied simultaneously. Thus $18$ is the sharp threshold for this method, while the problem remains open in higher dimensions.

It is worth noting that the torus enters the argument only through the obstruction to positive scalar curvature. Consequently, Theorem~\ref{thm:main} remains valid with $T^n$ replaced by any closed $n$-manifold that admits no metric of positive scalar curvature.

Question~\ref{q:petrunin} is one instance of the broader problem of understanding how extrinsically flat a closed immersed manifold can remain when confined to a bounded region. This problem has a substantial classical history~\cite{fary1950, chakerian1962, aminov1973, hasanis1979, jorge1981, jorge1981b}, and has received renewed attention through Gromov’s recent work on immersions with controlled curvature~\cite{gromov2022, gromov2022b, gromov2023, gromov2025}. Further contributions include work of Petrunin~\cite{petrunin2024, petrunin2024b}, Mendes~\cite{mendes2025}, Chodosh--Li~\cite{chodosh2026}, Chow--Wan~\cite{chow2026, chow2026II}, and the author~\cite{raffaelli2026}.

\section{Proof of Theorem~\ref{thm:main}}\label{sec:proof}

We prove the theorem for $5\leq n\leq 18$, since the case $n\leq 4$ is Petrunin's theorem~\cite[Thm.~1.2(d)]{petrunin2024}. We argue by contradiction. Recall that the torus admits no metric of positive scalar curvature~\cite{gromov1980}. Let $L$ be the conformal Laplacian of the induced metric $g$,
\begin{equation}\label{eq:laplacian}
L = -c_n\Delta+\scal,\qquad c_n= \frac{4(n-1)}{n-2},
\end{equation}
where $\Delta$ and $\scal$ denote the Laplace--Beltrami operator and scalar curvature of $g$, respectively. If $u>0$, then the conformally related metric
\begin{equation*}
\widehat g = u^{4/(n-2)} g
\end{equation*}
has scalar curvature
\begin{equation*}
\widehat \scal=u^{-(n+2)/(n-2)} Lu;
\end{equation*}
see \cite[p.~146]{aubin1998}. Consequently, there can be no positive smooth function $u$ on $T^n$ satisfying $Lu >0$ everywhere.

By Petrunin's spherical averaging identity,
\begin{equation*}
\lvert \two\rvert^2 =  \frac{n(n+2)\kappa - \lvert H\rvert^2}{2},
\end{equation*}
where $H=\tr \two$ is the mean curvature vector. Combining this identity with the Gauss equation
\begin{equation*}
\scal= \lvert H\rvert^2 - \lvert \two\rvert^2
\end{equation*}
gives
\begin{align*}
\scal &= \lvert H\rvert^2 - \frac{n(n+2)\kappa - \lvert H\rvert^2}{2}\\
&= \frac{3}{2} \lvert H\rvert^2 -\frac{1}{2}n(n+2)\kappa.
\end{align*}
Suppose, contrary to the conclusion of the theorem, that 
\begin{equation*}
\kappa(p) < \frac{3n}{n+2}
\end{equation*}
at every point $p\in T^n$. Define
\begin{equation*}
\delta= \frac{3n^2-n(n+2)\kappa}{2}.
\end{equation*}
Then $\delta$ is a positive smooth function, and 
\begin{equation}\label{eq:scal}
\scal=\frac{3}{2} (\lvert H\rvert^2 -n^2)+ \delta.
\end{equation}
We will use $\delta$ to construct a function $u>0$ for which $Lu>0$, obtaining the desired contradiction.

Let $X$ denote the immersion $T^n\looparrowright B^q$, and set $\rho = \lvert X\rvert^2/2$. Recall that 
\begin{equation*}
\nabla \rho =X^\top,\qquad \Delta \rho = n +\langle X,H\rangle = n + \langle X^\perp,H\rangle,
\end{equation*}
where $X^\top$ and $X^\perp$ denote the orthogonal projections of $X$ onto the tangent and normal bundles of $T^n$. For convenience in computing $\Delta u/u$, we take a radial function of the form 
\begin{equation*}
u= e^{-\psi\circ\rho},
\end{equation*}
with $\psi\colon\mathbb R\to\mathbb R$ smooth. Then
\begin{align}
\frac{\Delta u}{u} &= ((\psi'\circ\rho)^2 -\psi''\circ\rho ) \lvert \nabla\rho\rvert^2 -\psi'\circ\rho\,\Delta\rho \notag\\
&=((\psi'\circ\rho)^2 -\psi''\circ\rho ) \lvert X^\top\rvert^2 -\psi'\circ\rho\,(n +\langle X^\perp,H\rangle).\label{eq:delta}
\end{align}
Substituting \eqref{eq:scal} and \eqref{eq:delta} into \eqref{eq:laplacian}, we find
\begin{equation*}
\frac{Lu}{u} = \delta + \frac{3}{2} (\lvert H\rvert^2 -n^2)+c_n\psi'\circ\rho\,(n +\langle X^\perp,H\rangle)+c_n(\psi''\circ\rho - (\psi'\circ\rho)^2) \lvert X^\top\rvert^2.
\end{equation*}
Completing the square in $H$, we obtain
\begin{equation*}
\frac{3}{2}\lvert H\rvert^2 + c_n\psi'\circ\rho\,\langle X^\perp,H\rangle = \frac{3}{2}\mleft\lvert H + \frac{c_n\psi'\circ\rho}{3}X^\perp\mright\rvert^2 - \frac{(c_n\psi'\circ\rho)^2}{6}\lvert X^\perp\rvert^2.
\end{equation*}
Thus
\begin{equation}\label{eq:laplacian2}
\frac{Lu}{u} = \delta + \frac{3}{2}\mleft\lvert H + \frac{c_n\psi'\circ\rho}{3}X^\perp\mright\rvert^2  + R_n,
\end{equation}
where
\begin{equation*}
R_n = n c_n\psi'\circ\rho -\frac{3}{2}n^2 - \frac{(c_n\psi'\circ\rho)^2}{6}\lvert X^\perp\rvert^2+c_n(\psi''\circ\rho - (\psi'\circ\rho)^2) \lvert X^\top\rvert^2.
\end{equation*}
Since the first two terms in \eqref{eq:laplacian2} are respectively positive and nonnegative, it is enough to choose $\psi$ so that $R_n\geq 0$ everywhere.

To express this requirement independently of the immersion, we interpret the functions $\rho =\lvert X\rvert^2/2$ and $\lvert X^\top\rvert^2$ as variables describing their possible values. In fact, we set
\begin{equation*}
s= \frac{1-\lvert X\rvert^2}{2},\qquad t=\lvert X^\top\rvert^2.
\end{equation*}
This choice places the sharp configuration, realized by Gromov's flat tori, at the corner $(s,t)=(0,0)$. Since the torus lies in the unit ball,
\begin{equation}\label{eq:region}
0\leq s\leq \frac{1}{2}, \qquad \quad 0\leq t \leq 1-2s
\end{equation}
and
\begin{equation*}
\lvert X^\perp\rvert^2 = 1-2s-t.
\end{equation*}
Let $\phi(s)= \psi'(1/2-s)$. The remainder becomes 
\begin{equation*}
R_n(s,t)=n c_n\phi(s) -\frac{3}{2}n^2 - \frac{c_n^2\phi(s)^2}{6}(1-2s -t)-c_n (\phi'(s) + \phi(s)^2 ) t.
\end{equation*}
Thus, it suffices to find $\phi$ such that $R_n(s,t)\geq 0$ for every $(s,t)$ in the triangular region~\eqref{eq:region}. Note that any such $\phi$ must be positive, for otherwise $R_n(s,0)<0$. Moreover, at the corner $(s,t)=(0,0)$,
\begin{equation*}
R_n(0,0)= -\frac{1}{6}(c_n\phi(0)-3n)^2,
\end{equation*}
so any admissible $\phi$ must satisfy $c_n\phi(0)=3n$.

We now recast the requirement $R_n\geq 0$ in a more useful form. Since any admissible $\phi$ is positive, we may restrict to $\phi>0$ and divide by $c_n^2\phi^2/6$:
\begin{equation*}
\frac{6R_n(s,t)}{c_n^2\phi(s)^2}= \frac{6 n }{c_n\phi(s)} -\frac{9n^2}{c_n^2\phi(s)^2} - 1 + 2s +t - \frac{6 \phi'(s)}{c_n\phi(s)^2}t - \frac{6}{c_n} t.
\end{equation*}
Completing the square,
\begin{equation*}
\frac{6 n }{c_n\phi(s)} -\frac{9n^2}{c_n^2\phi(s)^2} =1- \mleft(\frac{3n}{c_n\phi(s)}-1 \mright)^2,
\end{equation*}
we obtain
\begin{equation*}
\frac{6R_n(s,t)}{c_n^2\phi(s)^2}= - \mleft(\frac{3n}{c_n\phi(s)}-1 \mright)^2  + 2s +t - \frac{6 \phi'(s)}{c_n\phi(s)^2}t - \frac{6}{c_n} t.
\end{equation*}
This suggests introducing the function
\begin{equation*}
f_n=\frac{3n}{c_n\phi}-1.
\end{equation*}
Under this change of variables, positivity of $\phi$ is equivalent to $f_n>-1$, with the inverse relation
\begin{equation*}
\phi=\frac{3n}{c_n(1+f_n)}.
\end{equation*}
Since $f_n' = -3n\phi'/(c_n\phi^2)$, substitution yields
\begin{align*}
R_n(s,t) &=\frac{3n^2}{2(1+f_n(s))^2}\mleft(-f_n(s)^2 +2s+t +\frac{2f'_n(s)}{n}t -\frac{6}{c_n}t\mright)\\
&= \frac{3n^2}{2(1+f_n(s))^2}\mleft(2s-f_n(s)^2 +\frac{2t}{n}\mleft(f_n'(s)-\frac{n(n-4)}{4(n-1)}\mright) \mright).
\end{align*}
Thus, setting 
\begin{equation*}
\lambda_n = \frac{n(n-4)}{4(n-1)}
\end{equation*}
and
\begin{equation*}
F_n(s,t) =2s-f_n(s)^2 +\frac{2t}{n}(f_n'(s)-\lambda_n),
\end{equation*}
our problem reduces to constructing $f_n>-1$ on $[0,1/2]$ such that $F_n(s,t)\geq 0$ for every admissible pair $(s,t)$.

For fixed $s$, the function $F_n(s,\cdot)$ is affine in $t$, so it is enough to check the two endpoints of the interval $0\leq t\leq 1-2s$:
\begin{equation}\label{eq:endpoint1}
F_n(s,0) = 2s-f_n(s)^2 \geq 0
\end{equation}
and 
\begin{equation}\label{eq:endpoint2}
F_n(s,1-2s) = 2s-f_n(s)^2 +\frac{2(1-2s)}{n}(f_n'(s)-\lambda_n)\geq 0.
\end{equation}
These conditions compete with one another: the first requires 
\begin{equation*}
\lvert f_n(s) \rvert\leq \sqrt{2s},
\end{equation*}
while the second forces $f_n$ to grow sufficiently quickly. In particular, for $s\in[0,1/2)$, \eqref{eq:endpoint2} is equivalent to
\begin{equation*}
f_n'(s) \geq \lambda_n - \frac{n}{2(1-2s)} ( 2s-f_n(s)^2).
\end{equation*}
If \eqref{eq:endpoint1} also holds, then the right-hand side is at most $\lambda_n$.

It remains to construct $f_n$. At $s=0$, \eqref{eq:endpoint1} gives $-f_n(0)^2\geq 0$, hence $f_n(0)=0$. Condition~\eqref{eq:endpoint2} then yields $f_n'(0)\geq \lambda_n$. The simplest candidate is therefore $f_n(s) =\lambda_n s$. This function, however, grows too rapidly, as the inequality $f_n(s)^2\leq 2s$ already fails when $n=12$. We thus consider the rational ansatz
\begin{equation*}
f_n(s)=\frac{\lambda_n s}{as^2+bs+1},
\end{equation*}
where $a,b$ are chosen so that the denominator $d(s)=as^2+bs+1$ is nonzero and $f_n'(s) \leq \lambda_n$ on $[0,1/2]$. For this choice, $f_n >0$ and 
\begin{equation*}
f_n'(s)= \lambda_n\frac{1-as^2}{d(s)^2}.
\end{equation*}
Moreover, for each fixed $s\in[0,1/2]$, both $F_n(s,0)$ and $F_n(s,1-2s)$ are nonincreasing in $n$ for $5\leq n\leq 18$. Indeed, $\lambda_n$ and $\lambda_n/n$ are increasing, while $f'_n/\lambda_n \leq 1$ is independent of $n$, so the claim follows directly from \eqref{eq:endpoint1}--\eqref{eq:endpoint2}. It therefore suffices to find suitable constants $a,b$ for which both inequalities hold when $n=18$; the same choice then works for every $5\leq n\leq 18$.

We first determine a candidate $a$. Substituting $f_n'(s)$ into \eqref{eq:endpoint2}, we obtain
\begin{equation*}
F_n(s,1-2s) = \frac{s}{d(s)^2}P_n (s),
\end{equation*}
where
\begin{equation*}
P_n (s) = 2d(s)^2 -\lambda_n^2 s- \frac{2\lambda_n}{n}(1-2s)(2b+(b^2+3a)s+2abs^2+a^2s^3 ).
\end{equation*}
A computation of $\partial P_n/\partial b$ reveals the distinguished point 
\begin{equation*}
s_n = \frac{n-4}{6(n-2)}
\end{equation*}
at which $P_n$ is independent of $b$. For $n=18$, we also have
\begin{equation*}
P_{18}(0) = 2-\frac{2\lambda_{18}}{18} 2b = 2 -\frac{14}{17}b \geq 0,
\end{equation*}
and hence $b \leq 17/7$. On the other hand, the condition $F_n(1/2,0)\geq 0$ gives
\begin{equation}\label{eq:ab}
\frac{a}{4}+\frac{b}{2} +1 \geq \frac{\lambda_{18}}{2}= \frac{63}{34}.
\end{equation}
Combining this with $b \leq 17/7$ yields
\begin{equation*}
a \geq -\frac{172}{119} \approx -1.445.
\end{equation*}
For an upper bound on $a$, we turn to the distinguished point $s_{18}=7/48$, where
\begin{equation*}
P_{18}(s_{18}) =- \frac{14161a+936}{332928},
\end{equation*}
so the requirement $P_{18}(s_{18})\geq 0$ amounts to
\begin{equation*}
a \leq -\frac{936}{14161} \approx -0.0661.
\end{equation*}
Thus $a=-1$ is a convenient choice. With $a=-1$, \eqref{eq:ab} reduces to $b\geq 75/34$. Together with $b\leq 17/7$, this gives
\begin{equation*}
\frac{75}{34}\leq b\leq \frac{17}{7}.
\end{equation*}
Thus we may set $b =9/4$. Note that for these choices,
\begin{equation*}
d(s) = 1+\frac{9}{4} s -s^2 \geq 1+s
\end{equation*}
on $[0,1/2]$, and hence
\begin{equation*}
f'_n(s)=\lambda_n\frac{1+s^2}{d(s)^2}\leq \lambda_n,
\end{equation*}
as desired.

It remains to verify that the resulting function $f_{18}$ satisfies both \eqref{eq:endpoint1} and \eqref{eq:endpoint2} for every $s\in [0,1/2]$. Since $F_{18}(0,0)=0$, condition~\eqref{eq:endpoint1} reduces, for $s\in(0,1/2]$, to
\begin{equation*}
\lambda_{18}^2 s\leq 2d(s)^2.
\end{equation*}
Because $d(s)>0$ on $[0,1/2]$, this is equivalent to
\begin{equation*}
\frac{d(s)}{\sqrt s} \geq \frac{\lambda_{18}}{\sqrt{2}}.
\end{equation*}
The function $d(s)/\sqrt s$ is decreasing on $(0,1/2]$, so it is enough to check the inequality at $s=1/2$. There it follows from $d(1/2)\geq \lambda_{18}/2$. Hence \eqref{eq:endpoint1} holds throughout $[0,1/2]$. 

We finally verify \eqref{eq:endpoint2}. A direct computation gives
\begin{equation*}
F_{18}(s,1-2s) = \frac{s}{4624 d(s)^2}P_{18}(s),
\end{equation*}
where 
\begin{equation*}
P_{18}(s)= 13056s^4-60656s^3+44744s^2-8679s+680.
\end{equation*}
Thus it suffices to show that $P_{18}(s)\geq 0$ on $[0,1/2]$. Since $P_{18}(1/2)=1521/2>0$, consider $s\in [0,1/2)$ and set
\begin{equation*}
r= \frac{2s}{1-2s}\geq 0.
\end{equation*}
Then
\begin{equation*}
2(1+r)^4 P_{18}\mleft(\frac{r}{2(1+r)}\mright) = 1521r^4 + 8983 r^3 + 4495 r^2 - 3239r + 1360.
\end{equation*}
The quadratic $4495 r^2 - 3239r + 1360$ has negative discriminant and positive leading coefficient, hence is positive for every $r\geq 0$. All remaining terms are nonnegative there, so $P_{18}(s) >0$ on $[0,1/2)$. Together with $P_{18}(1/2)>0$, this proves \eqref{eq:endpoint2}. The preceding reductions thus yield the desired contradiction.

\section{The $n=18$ threshold}\label{sec:18}

Here we explain why $n=18$ is the largest dimension for which this method can work. The obstruction does not come from the particular ansatz used above, but already from the reduction to the one-dimensional conditions \eqref{eq:endpoint1}--\eqref{eq:endpoint2}. 

\begin{proposition}\label{prop:main}
For $n\geq 19$, there does not exist a function $f \in\mathcal C^1([0,1/2))$ with $f(0)=0$ such that
\begin{equation}\label{eq:inequality}
f'(s)\geq \lambda_n-\frac{n}{2(1-2s)}(2s-f(s)^2)\qquad \forall s\in[0,1/2).
\end{equation}
\end{proposition}

The proof of Proposition~\ref{prop:main} is based on the following elementary one-sided Sturm comparison lemma. Its proof is the classical Wronskian argument for singular Sturm comparison; cf.~\cite[proof of Thm.~1(i)]{aharonov2010}. We will also use some basic facts about Bessel functions, for which we refer the reader to \cite{olver2010}.

\begin{lemma}[Sturm comparison]\label{lem:sturm}
Let $\gamma>0$, and let $V,\widetilde V$ be continuous on $(0,\gamma]$, with $V < \widetilde V$. Suppose that $y,\tilde y$ solve
\begin{equation*}
(xy')'+Vy=0, \qquad (x\tilde y')'+\widetilde V\tilde y=0
\end{equation*}
on $(0,\gamma]$, with
\begin{equation*}
y(\gamma)=\tilde y(\gamma)=1,\qquad y'(\gamma)=\tilde y'(\gamma)=0.
\end{equation*}
If
\begin{equation*}
0<y\leq 1\qquad \text{on }(0,\gamma],
\end{equation*}
then $\tilde y$ has a zero in $(0,\gamma)$.
\end{lemma}

\begin{proof}
Suppose, towards a contradiction, that $\tilde y$ has no zero in $(0,\gamma)$; since $\tilde y(\gamma)=1$, we have $\tilde y>0$ on $(0,\gamma]$. Define 
\begin{equation*}
W=x(y\tilde y'-y'\tilde y),
\end{equation*}
so that 
\begin{equation*}
\mleft(\frac{\tilde y}{y}\mright)'=\frac{W}{xy^2}.
\end{equation*}
Using the equations for $y$ and $\tilde y$, we then obtain
\begin{equation*}
W'=(V-\widetilde V)y\tilde y<0.
\end{equation*}
Since $W(\gamma)=0$, it follows that $W>0$ on $(0,\gamma)$.

Fix $x_0\in(0,\gamma)$. Then, for every $x\in(0,x_0)$, since $W$ is decreasing,
\begin{equation*}
\mleft(\frac{\tilde y}{y}\mright)'\geq \frac{W(x_0)}{xy^2}\geq \frac{W(x_0)}{x}.
\end{equation*}
Integrating from $x$ to $x_0$ gives
\begin{equation*}
\frac{\tilde y(x_0)}{y(x_0)}-\frac{\tilde y(x)}{y(x)}\geq W(x_0)\log\frac{x_0}{x}.
\end{equation*}
The right-hand side tends to $+\infty$ as $x\to0^+$, whereas the left-hand side is bounded above by $\tilde y(x_0)/y(x_0)$, because $\tilde y/y>0$. This is a contradiction. Hence $\tilde y$ has a zero in $(0,\gamma)$.
\end{proof}

\begin{proof}[Proof of Proposition~\textup{\ref{prop:main}}]
Let $h$ be the maximal solution of
\begin{equation}\label{eq:riccati}
f'(s)= \lambda_n-\frac{n}{2(1-2s)}(2s-f(s)^2),\qquad f(0)=0.
\end{equation}
By the differential form of Grönwall's inequality, every function $f$ satisfying \eqref{eq:inequality} with $f(0)=0$ lies above $h$ on the interval where $h$ is defined. Consequently, it is enough to show that $h$ blows up to $+\infty$ at some $s<1/2$.

Equation~\eqref{eq:riccati} is a Riccati equation. The usual logarithmic derivative substitution, together with a suitable change of variable, transforms it into a Bessel equation. For our purposes, it is more convenient to use this transformation in the reverse direction. More precisely, set 
\begin{equation*}
\gamma_n= n\sqrt\frac{3(n-2)}{8(n-1)},\qquad \nu=\frac{n}{2},
\end{equation*}
and let $\tilde y$ be the solution of the Bessel equation
\begin{equation}\label{eq:bessel}
(x\tilde y')' +\mleft(x-\frac{\nu^2}{x} \mright)\tilde y=0,
\end{equation}
with
\begin{equation*}
\tilde y(\gamma_n)=1,\qquad \tilde y'(\gamma_n)=0.
\end{equation*}
As long as $\tilde y$ does not vanish, set
\begin{equation*}
x=\gamma_n\sqrt{1-2s},\qquad H(s) = \frac{2x}{n} \frac{\tilde y'(x)}{\tilde y(x)}.
\end{equation*}
Dividing \eqref{eq:bessel} by $\tilde y$, using
\begin{equation*}
\mleft( \frac{\tilde y'}{\tilde y}\mright)' = \frac{\tilde y''}{\tilde y} - \mleft( \frac{\tilde y'}{\tilde y}\mright)^2,\qquad \frac{dx}{ds}=-\frac{\gamma_n^2}{x},
\end{equation*} 
and substituting
\begin{equation*}
\frac{\tilde y'}{\tilde y} = \frac{nH}{2x},
\end{equation*}
gives
\begin{equation*}
H'(s)= \frac{2\gamma_n^2}{n} + \frac{n}{2(1-2s)}(H(s)^2-1).
\end{equation*}
Since $2\gamma_n^2/n = \lambda_n+n/2$, this is exactly \eqref{eq:riccati}. Moreover, $H(0)=0$, so uniqueness gives $H=h$ as long as $\tilde y$ remains nonzero. A nontrivial solution of \eqref{eq:bessel} cannot have a double zero, so any zero of $\tilde y$ is simple. It therefore suffices to show that $\tilde y$ has a zero in $(0,\gamma_n)$.

Let $J_\nu$ be the Bessel function of the first kind, which is a solution of the same Bessel equation on $(0,\infty)$, and let $j'= j'_{\nu,1}$ be the first positive zero of $J'_\nu$. We claim that if $\gamma_n>j'$, then $\tilde y$ has a zero in $(0,\gamma_n)$. Set $\alpha=j'/\gamma_n<1$, and define
\begin{equation*}
y(x)=\frac{J_\nu(\alpha x)}{J_\nu(j')}.
\end{equation*}
Then
\begin{equation*}
(x y')' +\mleft(\alpha^2 x-\frac{\nu^2}{x} \mright) y=0
\end{equation*}
and, since $\alpha\gamma_n=j'$,
\begin{equation*}
y(\gamma_n)=1,\qquad y'(\gamma_n)=0.
\end{equation*}
Moreover, $J_\nu$ is positive and strictly increasing on $(0,j']$. Hence
\begin{equation*}
0<y(x)\leq 1\qquad \text{for }0<x\leq \gamma_n.
\end{equation*}
We may therefore apply Lemma~\ref{lem:sturm} with
\begin{equation*}
V(x)=\alpha^2 x-\frac{\nu^2}{x},\qquad \widetilde V(x)= x-\frac{\nu^2}{x}.
\end{equation*}
Indeed,
\begin{equation*}
\widetilde V(x)-V(x) =(1-\alpha^2)x>0 \qquad \text{for }0<x<\gamma_n.
\end{equation*}
The lemma shows that $\tilde y$ has a zero in $(0,\gamma_n)$.

It remains to verify that $\gamma_n>j'$ for $n\geq 19$. By a theorem of McCann~\cite{mccann1977}, the function $\nu\mapsto j'_{\nu,1}/\nu$ is strictly decreasing. Therefore
\begin{equation*}
\frac{j'_{n/2,1}}{n}=\frac{1}{2}\frac{j'_{\nu,1}}{\nu}
\end{equation*}
is strictly decreasing in $n$. On the other hand, 
\begin{equation*}
\frac{\gamma_n}{n} = \sqrt\frac{3(n-2)}{8(n-1)}= \sqrt{\frac{3}{8} \left(1- \frac{1}{n-1}\right)}
\end{equation*}
is strictly increasing in $n$. The two quantities cross between $n=18$ and $n=19$, since
\begin{equation*}
\gamma_{18} = 18\sqrt{\frac{6}{17}} \approx 10.69359 < 10.71143 \approx j'_{9,1},
\end{equation*}
whereas
\begin{equation*}
\gamma_{19} = 19\sqrt{\frac{17}{48}} \approx 11.30726 > 11.24168 \approx j'_{19/2,1}.
\end{equation*}
It follows that $\gamma_n > j'_{n/2,1}$ for every $n\geq 19$. This proves the proposition.
\end{proof}
\bibliographystyle{amsplain}
\bibliography{references}

\providecommand{\bysame}{\leavevmode\hbox to3em{\hrulefill}\thinspace}
\providecommand{\MR}{\relax\ifhmode\unskip\space\fi MR }
\providecommand{\MRhref}[2]{%
  \href{http://www.ams.org/mathscinet-getitem?mr=#1}{#2}
}
\providecommand{\href}[2]{#2}
\begin{thebibliography}{10}

\bibitem{aharonov2010}
Dov Aharonov and Uri Elias, \emph{Singular {S}turm comparison theorems}, J.
  Math. Anal. Appl. \textbf{371} (2010), no.~2, 759--763. \MR{2670153}

\bibitem{aminov1973}
Ju.~A. Aminov, \emph{The exterior diameter of an immersed {R}iemannian
  manifold}, Math. USSR-Sb. \textbf{21} (1973), no.~3, 449--454.

\bibitem{aubin1998}
Thierry Aubin, \emph{Some nonlinear problems in {R}iemannian geometry},
  Springer Monographs in Mathematics, Springer-Verlag, Berlin, 1998.
  \MR{1636569}

\bibitem{chakerian1962}
G.~D. Chakerian, \emph{An inequality for closed space curves}, Pacific J. Math.
  \textbf{12} (1962), 53--57. \MR{140000}

\bibitem{chodosh2026}
Otis Chodosh and Chao Li, \emph{Immersions with small normal curvature},
  arXiv:2602.15728, 2026.

\bibitem{chow2026}
Tsz-Kiu~Aaron Chow and Jingbo Wan, \emph{Maximal normal curvature and
  {V}eronese rigidity}, arXiv:2607.00949, 2026.

\bibitem{chow2026II}
\bysame, \emph{Small normal curvature and three-manifold topology},
  arXiv:2608.18002, 2026.

\bibitem{fary1950}
Istv\'an F\'ary, \emph{Sur certaines in\'egalit\'es g\'eom\'etriques}, Acta
  Sci. Math. (Szeged) \textbf{12} (1950), 117--124. \MR{38090}

\bibitem{gromov1980}
Mikhael Gromov and H.~Blaine Lawson, Jr., \emph{Spin and scalar curvature in
  the presence of a fundamental group. {I}}, Ann. of Math. (2) \textbf{111}
  (1980), no.~2, 209--230. \MR{569070}

\bibitem{gromov2022}
Misha Gromov, \emph{Curvature, {K}olmogorov diameter, {H}ilbert rational
  designs and overtwisted immersions}, arXiv:2210.13256, 2022.

\bibitem{gromov2023}
\bysame, \emph{Isometric immersions with controlled curvatures}, J. Assoc.
  Math. Res. \textbf{1} (2023), no.~1, 1--15. \MR{4850217}

\bibitem{gromov2025}
\bysame, \emph{Lectures on immersions with controlled curvatures},
  arXiv:2511.01796, 2025.

\bibitem{gromov2022b}
\bysame, \emph{Scalar curvature, injectivity radius and immersions with small
  second fundamental forms}, J. Assoc. Math. Res. \textbf{3} (2025), no.~1,
  27--71. \MR{4876228}

\bibitem{hasanis1979}
Th. Hasanis and D.~Koutroufiotis, \emph{Immersions of bounded mean curvature},
  Arch. Math. (Basel) \textbf{33} (1979), no.~2, 170--171. \MR{557750}

\bibitem{jorge1981}
L.~Jorge and D.~Koutroufiotis, \emph{An estimate for the curvature of bounded
  submanifolds}, Amer. J. Math. \textbf{103} (1981), no.~4, 711--725.
  \MR{623135}

\bibitem{jorge1981b}
Luqu\'esio P. de~M. Jorge and Frederico~V. Xavier, \emph{An inequality between
  the exterior diameter and the mean curvature of bounded immersions}, Math. Z.
  \textbf{178} (1981), no.~1, 77--82. \MR{627095}

\bibitem{mccann1977}
Roger~C. McCann, \emph{Inequalities for the zeros of {B}essel functions}, SIAM
  J. Math. Anal. \textbf{8} (1977), no.~1, 166--170. \MR{422722}

\bibitem{mendes2025}
Ricardo A.~E. Mendes, \emph{Diameter and focal radius of submanifolds}, Ann.
  Global Anal. Geom. \textbf{68} (2025), no.~2, Paper No. 4, 6. \MR{4927774}

\bibitem{olver2010}
F.~W.~J. Olver and L.~C. Maximon, \emph{Bessel functions}, N{IST} handbook of
  mathematical functions, U.S. Dept. Commerce, Washington, 2010, pp.~215--286.
  \MR{2655350}

\bibitem{petrunin2024}
Anton Petrunin, \emph{Gromov's tori are optimal}, Geom. Funct. Anal.
  \textbf{34} (2024), no.~1, 202--208. \MR{4706446}

\bibitem{petrunin2024b}
\bysame, \emph{Veronese minimizes normal curvatures}, arXiv:2408.05909, 2024.

\bibitem{raffaelli2026}
Matteo Raffaelli, \emph{Normal curvature bounds for immersions into
  {R}iemannian domains}, arXiv:2606.03659, 2026.

\end{thebibliography}
\end{document}